\documentclass[12pt]{amsart}
\usepackage{euscript}
\usepackage{amssymb}
\usepackage{color}
\usepackage{ulem}

\def\semidirect{\rtimes}
\def\Spec{\mathop{\mathrm{Spec}}\nolimits}
\def\Frac{\mathop{\mathrm{Frac}}\nolimits}
\def\script#1{\EuScript{#1}}
\def\cp{\C P}
\def\H{\mathrm{H}}
\def\Spf{\mathop{\mathrm{Spf}}\nolimits}

\def\PGamma{P\Gamma}
\def\dihedral{D}
\def\WW{\mathfrak{W}}
\def\an{{\rm an}}
\def\BT{\script{B}}
\def\Pitilde{\tilde{\Pi}}

\def\Omegahat{\widehat{\Omega}}

\def\lambdabar{\bar{\lambda}}
\def\P{\mathbb{P}}
\def\Q{\mathbb{Q}}
\def\C{\mathbb{C}}
\def\R{\mathbb{R}}
\def\Z{\mathbb{Z}}
\def\F{\mathbb{F}}
\def\tensor{\otimes}
\let\iso\cong
\let\cong\equiv
\def\calO{\script{O}}

\def\PSL{{\rm PSL}}
\def\PGL{{\rm PGL}}

\def\GL{{\rm GL}}
\def\PO{{\rm PO}}

\def\sset{\subseteq}

\def\Berk{{\rm Berk}}
\def\Qbar{\overline{\Q}}
\def\Kbar{K\llap{$\overline{\phantom{\rm K}}$}}
\def\Kbarscript{K\llap{$\overline{\phantom{\scriptstyle\rm K}}$}}

\newtheorem{theorem}{Theorem}
\newtheorem{prop}[theorem]{Proposition}

\newtheorem{lemma}[theorem]{Lemma}
\newtheorem{lem}[theorem]{Lemma}

\renewcommand\theenumi{\roman{enumi}}
\theoremstyle{remark}
\newtheorem*{remark}{Remark}

\numberwithin{theorem}{section}
\numberwithin{equation}{section}
\numberwithin{figure}{section}

\begin{document}

\title[A fake plane via 2-adic uniformization with torsion]{A fake projective plane via 2-adic uniformization with torsion}
\author{Daniel Allcock}
\address{Department of Mathematics\\University of Texas, Austin}
\email{allcock@math.utexas.edu}
\urladdr{http://www.math.utexas.edu/\textasciitilde allcock}
\thanks{First author supported by NSF grant DMS-1101566.}
\author{Fumiharu Kato}
\address{Department of Mathematics\\Kumamoto University}
\email{kato@sci.kumamoto-u.ac.jp}
\urladdr{http://www.sci.kumamoto-u.ac.jp/\textasciitilde kato/index-e.html}

\date{November 2, 2014}

\begin{abstract}
We adapt the theory of non-Archimedean uniformization to construct a
smooth surface from a lattice in $\PGL_3(\Q_2)$ that has
nontrivial torsion.  It turns out to be a fake projective plane,  commensurable with
Mumford's fake plane yet distinct from it and the other
fake planes that arise from $2$-adic uniformization by
torsion-free groups.  As part of the proof, and of independent
interest, we compute the homotopy type of the Berkovich space of our
plane.
\end{abstract}

\maketitle

\noindent
The original definition of a fake projective plane is
a compact complex surface  that has the
same Betti numbers as $\cp^2$, but is not $\cp^2$.  The first example was
given by Mumford \cite{Mum1}, and all fake planes have recently been
classified by Prasad-Yeung \cite{PY} and Cartwright-Steger
\cite{CS}: there are 100 of them up to isomorphism, in 50
complex-conjugate pairs.

Mumford used the theory of $2$-adic uniformization, beginning with a
well-chosen discrete subgroup of $\PGL_3(\Q_2)$.  His construction
yields a fake projective plane over $\Q_2$.  For this to make sense, we
use Mumford's definition of a fake plane $X$ over a
general field $K$, which specializes to the above definition when
$K=\C$.  Namely: $X$ is a smooth and geometrically connected proper
surface 
over $K$, such that its base change to $X_{\Kbarscript}$
satisfies $P_g=q=0$, $c^2_1=3c_2=9$ and has ample canonical class.
Here $\Kbar$ denotes the algebraic closure of $K$.
To get a fake plane in the original sense, one identifies $\Qbar_2$
with $\C$ by some isomorphism.

The machinery used by Mumford required his discrete subgroup of
$\PGL_3(\Q_2)$ to be torsion-free, and there are exactly two
additional fake planes that can be constructed this way \cite{IK98}.
The purpose of this paper is to show that torsion can be allowed in
the construction, leading to a ``new'' fake plane.  Of course, it
occurs in the Prasad-Yeung--Cartwright-Steger enumeration; what is
new is that there is another fake plane realizable by $2$-adic
uniformization.

This is interesting for two reasons.  First, uniformization by groups
containing torsion is possible and useful.  Second, in the $2$-adic
approach, $X$ is the generic fiber of a flat family over the $2$-adic
integers $\Z_2$, and the central fiber gives a great deal of geometric
information about $X$ that is not available in the Prasad-Yeung
approach.  For example, Ishida \cite{Ishida} showed that Mumford's
fake plane covers an elliptic surface whose singular fibers have
specific types, and Keum was able to use this to construct another
fake plane \cite{Keum}.  The main open problem about fake planes is to
construct one by non-transcendental methods.  Since $2$-adic
uniformization yields additional information about the planes that may
be constructed using it, we may reasonably hope that it will help
solve this problem.

\section{Non-Archimedean uniformization}
\label{sec-uniformization}

\noindent
In this section we give background material on non-Archimedean
uniformization and recall how this guided Mumford in choosing the
tor\-sion-free lattice in $\PGL_3(\Q_2)$ that uniformizes
his fake plane.  
We call his lattice $\Sigma_M$; his notation was $\Gamma$.
In the next section we will describe another lattice
$\Sigma_L\sset\PGL_3(\Q_2)$ and show how to use it to build a fake
plane, even though $\Sigma_L$ contains torsion.

Let $R$ be a complete discrete valuation ring, $K$ its field of
fractions and $k=R/\pi R$ the residue field, where $\pi\in R$ is a
fixed uniformizer.  We assume $k$ is finite with say $q$ elements.  We
write $\BT_K$ for the Bruhat--Tits building of $\PGL_3(K)$.  This is a
$2$-dimensional simplicial complex whose vertices are the homothety
classes of rank-three $R$-submodules of $K^3$.  Vertices are joined by
an edge if (after scaling) one module contains the other with quotient
a $1$-dimensional $k$ vector space.  Three vertices span a
triangle if they are pairwise joined by edges.   $\PGL_3(K)$
acts on $\BT_K$ in the obvious way.

The Drinfeld upper-half plane $\Omega^2_K$ over $K$ means the set of
closed points of $\P^2_K$, minus those that lie on
$K$-rational lines.  It is an 
admissible
open subset of the rigid analytic space
$\P^{2,\an}_K$, 
hence a rigid analytic space itself.  We write
$\Omegahat^2_K$ for the `standard' formal model of $\Omega^2_K$ from
\cite[Prop.\ 2.4]{Mus1}, where it is denoted 
$\script{P}(\Delta_{\ast})$ with $\Delta_{\ast}=\BT_K$.  This is a
formal scheme, flat and locally of finite type over $\Spf R$, and equipped
with a $\PGL_3(K)$-action.  It has the following properties.
\begin{itemize}
\item The closed fiber $\Omegahat^2_{K,0}$ is normal crossing, with
  each component a non-singular rational surface over $k$,  isomorphic
  to $\P^2_k$ blown up at all
  $k$-rational points.
\item The double curves of $\Omegahat^2_{K,0}$ that lie in one of these components are the exceptional
  curves of this blowup, which are $(-1)$-curves, and the proper
  transforms of $k$-rational lines of $\P^2_k$, which are
  $(-q)$-curves.    Each double curve has
  different self-intersection numbers in the two components containing
  it.
\item The dual complex of the closed fiber $\Omegahat^2_{K,0}$ is
$\PGL_3(K)$-equi\-var\-iant\-ly isomorphic to $\BT_K$.
\end{itemize}

\noindent
The second property allows us to orient the edges of $\BT_K$, a
property we will use only in section~\ref{sec-distinctness}.  An edge
corresponds to a curve where two components of $\Omegahat^2_K$ meet;
we orient the edge so that it goes from the component in which the
curve has self-intersection $-1$ to the one in which it has
self-intersection $-q$.  The mnemonic is that the arrow on the edge
can be thought of as a greater-than sign, indicating $-1>-q$.
Obviously $\PGL_3(K)$ respects the orientations of edges.
A triangle in $\BT_K$ corresponds to an intersection point of $3$
components of $\Omegahat^2_K$, and   from the description of the
double curves it
is easy to see that the edges corresponding to the three incident
double curves form an oriented
circuit.  This induces a cyclic ordering on the set of these double curves.
(Everything in this paragraph could 
alternately be developed in terms of $R$-submodules of $K^3$
containing each other.)

Now suppose $\Gamma$ is a torsion-free lattice in $\PGL_3(K)$; all
lattices are uniform, so the coset space is compact \cite{Tamagawa}.
  Because $\Gamma$ is discrete and torsion-free, it acts freely on
  $\BT_K$.  By the correspondence between the vertices of $\BT_K$ and
  the components of $\Omegahat^2_K$, $\Gamma$ also acts freely on
  $\widehat{\Omega}^2_K$, and properly discontinuously with respect to
  the Zariski topology.  The quotient
  $\widehat{\mathfrak{X}}_\Gamma:=\widehat{\Omega}^2_K/\Gamma$ is a
  proper flat formal $R$-scheme, whose closed fiber
  $\widehat{\mathfrak{X}}_{\Gamma,0}$ is a normal crossing divisor
  \cite[Thm.~3.1]{Mus1}.

Its relative dualizing sheaf
$\omega_{\widehat{\mathfrak{X}}_{\Gamma}/R}$ over $R$ is thus the
sheaf of relative log differential $2$-forms.  Since there are
`enough' double curves on each component one can show that
$\omega_{\widehat{\mathfrak{X}}_{\Gamma}/R}$ is relatively ample
(\cite[p.\ 204]{Mus1}).  This implies that the formal scheme
$\widehat{\mathfrak{X}}_{\Gamma}$ is algebraizable, that is,
isomorphic to the $\pi$-adic formal completion of a proper flat
$R$-scheme $\mathfrak{X}_{\Gamma}$, which is uniquely determined up to
isomorphism.  The generic fiber
$X_{\Gamma}:=\mathfrak{X}_{\Gamma,\eta}$ is then a proper smooth
surface over $K$, and has ample canonical class.  See
\cite[\S5.4]{EGA3} for background.

On the other hand, $\Gamma$ also acts freely and properly
discontinuously on $\Omega^2_K$.  This allows the construction of the
rigid analytic quotient $\Omega^2_K/\Gamma$, which turns out to be
$K$-isomorphic to the rigid analytic surface $X^{\an}_{\Gamma}$ got
from $X_\Gamma$ by analytification.  In other words,
$\Omega^2_K/\Gamma$
is the Raynaud generic fiber of the formal scheme
$\hat{\mathfrak{X}}_{\Gamma}$. 
In particular, the closed points of
$X_\Gamma$ are in bijection with those of $\Omega^2_K/\Gamma$.  

Now we come to Mumford's construction of his fake plane:

\begin{prop}[{\rm \cite[\S 1]{Mum1}}]
\label{prop-invariants}
Let $N$ be the number of\, $\Gamma$-orbits on the vertices of $\BT_K$,
and as usual write $q(X):=\dim\H^1(X,\calO_X)$ for the irregularity 
and $P_g(X):=\dim\H^2(X,\calO_X)$ 
for the geometric genus of $X=X_{\Gamma}$.
Then
\begin{itemize}
\item[{\rm (a)}] $\chi(\calO_X):=1-q(X)+P_g(X)$ is equal to $\frac{\,N\,}{3}(q-1)^2(q+1);$
\item[{\rm (b)}] $c^2_1(X)=3c_2(X)=3N(q-1)^2(q+1);$
\item[{\rm (c)}] $q(X)=0;$
\item[{\rm (d)}] the canonical class $K_X$ is ample.
\hfill\qed
\end{itemize}
\end{prop}

Mumford took $R=\Z_2$ (so $q=2$) and chose a lattice in $\PGL_3(\Q_2)$
we call $\Sigma_M$, which is vertex-transitive (so $N=1$) and
torsion-free (so the machinery applies).  Abbreviating $X_{\Sigma_M}$
to $X_M$, it follows that $X_M$ is a fake projective plane over
$\Q_2$.  

We will use the same idea, but more work is required because the group
$\Sigma_L$ uniformizing our 
fake plane $X_L$ contains torsion.  We
have now provided all the background necessary for the construction of our
plane, so the reader could skip to section~\ref{sec-construction} immediately.  

\medskip
By \cite{IK98} there are exactly three fake planes that can be
obtained from Mumford's construction, using torsion-free groups.  To
show that our fake plane is distinct from them, in section~\ref{sec-distinctness} we
will compare their Berkovich spaces.  Here is the necessary
background, cf.\ \cite{Berk1}\cite{Berk2}.

In the sequel, for a rigid space or an algebraic variety $Z$ over a
complete non-Archimedean field, we denote by $Z^{\Berk}$ the
associated Berkovich space; see \cite[1.6]{Berk2} for the relation
between rigid geometry and Berkovich geometry, and \cite[3.4]{Berk1}
for Berkovich GAGA.  Notice that, in both cases, the associated
Berkovich space $Z^{\Berk}$ is uniquely determined, and that the
functor $Z\mapsto Z^{\Berk}$ is fully faithful.

Let $X$ be a quasi-projective variety over $K$, and $G$ a finite group acting on $X$ by automorphisms over $K$.
It is well-known that the quotient $X/G$ is represented by a quasi-projective variety over $K$.
\begin{lemma}\label{lem-finitequot}
The quotient $X^{\Berk}/G$ by the canonically induced action of $G$ on $X^{\Berk}$ is represented by a Berkovich $K$-analytic space, and is naturally isomorphic to $(X/G)^{\Berk}$.
Moreover, the underlying topological space of $X^{\Berk}/G$ coincides with the topological quotient of the topological space $X^{\Berk}$ by $G$.
\end{lemma}

\begin{proof}
We may assume that $X$ is affine $X=\Spec A$, where $A$ is a finite type algebra over $K$.
By \cite[Remark 3.4.2]{Berk1}, we know that $X^{\Berk}$ is the set of all multiplicative seminorms $|\cdot|$ on $A$ that extends the valuation $\|\cdot\|$ on $K$.
Set $B=A^G$, the $G$-invariant part, which is again a finite type algebra over $K$.
Consider $Y=\Spec B$ and the map $\pi\colon X^{\Berk}\rightarrow
Y^{\Berk}$ given by the restriction of seminorms.

First we show that $\pi$ is surjective.
Take $y=|\cdot|_y\in Y^{\Berk}$, and let $\mathfrak{q}$ be the kernel of $|\cdot|_y$, which is a prime ideal of $B$.
Since $A/B$ is finite, there exists a prime ideal $\mathfrak{p}$ of $A$ such that $\mathfrak{p}\cap B=\mathfrak{q}$.
Let $\mathcal{H}(y)$ be the completion of the residue field $\kappa=\Frac(B/\mathfrak{q})$ by the valuation induced from $|\cdot|_y$.
Since $\kappa=\Frac(B/\mathfrak{q})\hookrightarrow\Frac(A/\mathfrak{p})$ is finite, $\mathcal{H}(y)\hookrightarrow\Frac(A/\mathfrak{p})\otimes_{\kappa}\mathcal{H}(y)$ is a finite extension of fields, and hence the valuation $|\cdot|_y$ uniquely extends to a valuation on the latter field.
We thus have a multiplicative seminorm $x=|\cdot|_x$ on $A$, which extends $|\cdot|_y$, which shows the surjectivity of $\pi$.

Let $x=|\cdot|_x\in X^{\Berk}$, and consider $x^g$ by $g\in G$, which
is the seminorm given by the composition
$A\stackrel{g}{\rightarrow}A\stackrel{|\cdot|_x}{\rightarrow}\R_{\geq
  0}$.  In this situation, we clearly have $x|_B=x^g|_B$.  Conversely,
suppose $x=\hbox{$|\cdot|$}_x$ and $x'=\hbox{$|\cdot|$}_{x'}$ are
points of $X^{\Berk}$ and satisfy $x|_B=x'|_B$.  Let $\mathfrak{p}$
and $\mathfrak{p}'$ be the kernels of $|\cdot|_x$ and $|\cdot|_{x'}$,
and $\mathfrak{q}$ the kernel of their restriction on $B$.  Since
$\mathfrak{q}=\mathfrak{p}\cap B=\mathfrak{p}'\cap B$, there exists
$g\in G$ such that $g^{-1}(\mathfrak{p})=\mathfrak{p'}$.  Replacing
$x'$ by $x^{\prime g}$, we may assume $\mathfrak{p}=\mathfrak{p}'$.
Then, by the uniqueness of the extension, $|\cdot|_x$ and $|\cdot|_{x'}$
coincide on $\Frac(A/\mathfrak{p})\otimes_{\kappa}\mathcal{H}(y)$, and
hence we have $x=x'$.

Thus the map $X^{\Berk}/G\rightarrow Y^{\Berk}$ is set-theoretically
bijective. By the construction, it is clearly continuous. Since
$X^{\Berk}$ is compact and $Y^{\Berk}$ is Hausdorff, we deduce that
$X^{\Berk}/G\rightarrow Y^{\Berk}$ gives a homeomorphism. Hence one
can endow $X^{\Berk}/G$ with the structure of a Berkovich strictly
$K$-analytic space induced from that of $Y^{\Berk}=(X/G)^{\Berk}$. It
is now clear that the resulting $K$-analytic space $X^{\Berk}/G$ gives
the quotient of $X^{\Berk}$ by $G$ in the category of Berkovich
$K$-analytic spaces.
\end{proof}

Let $\Gamma$ be a lattice in $\PGL_3(K)$.  (One could replace $3$ by
any $n$ by making trivial changes below.)  
By Selberg's lemma \cite{Alperin}
we know
there exists a torsion free normal subgroup $\Gamma_0\subseteq\Gamma$
of finite index.  Set $G=\Gamma/\Gamma_0$.  As discussed earlier in
this section, the quotient $\Omega^2_K/\Gamma_0$ is algebraizable, and
is of the form $X^{\an}_{\Gamma_0}$ by a smooth projective variety
$X_{\Gamma_0}$ over $K$, which is obtained as the generic fiber of the
algebraization $\mathfrak{X}_{\Gamma_0}$ of the formal scheme
$\widehat{\mathfrak{X}}_{\Gamma_0}=\Omegahat^2_K/\Gamma_0$. The rigid
analytic space $X^{\an}_{\Gamma_0}/G\cong\Omega^2_K/\Gamma$ is then
isomorphic to $(X_{\Gamma_0}/G)^{\an}$, hence is algebraized by the
projective variety $X_{\Gamma_0}/G$.  
We define $X_{\Gamma}$ as $X_{\Gamma_0}/G$.  It is
  independent of the choice of $\Gamma_0$ because if $\Gamma_0'$ were
  another torsion free normal subgroup of $\Gamma$ of finite index,
  then both $X_{\Gamma_0}/(\Gamma/\Gamma_0)$ and
  $X_{\Gamma_0'}/(\Gamma/\Gamma_0')$ are naturally identified with
$X_{\Gamma_0\cap\Gamma_0'}\!\bigm/\!\!\bigl(\Gamma/(\Gamma_0\cap\Gamma_0')\bigr)$

Now, let $K'/K$ be a finite extension.
In \cite{Berk3}, Berkovich constructed a natural $\PGL_3(K)$-equivariant retraction $\tau\colon\Omega^{2,\Berk}_K\otimes_K K'\rightarrow \BT_K$.
Moreover, as is well-known, there exists a natural $\PGL_3(K)$-equivariant homotopy between the identity map of $\Omega^{2,\Berk}_K\otimes_K K'$ and the retraction map $\tau$ (cf.\ \cite[Remark 5.12 (iii)]{Berk4}).
Since the quotient map $\Omega^{2,\Berk}_K\rightarrow\Omega^{2,\Berk}_K/\Gamma_0$ is a topological covering map, we have the induced deformation retract $X^{\Berk}_{\Gamma_0}=\Omega^{2,\Berk}_K/\Gamma_0\rightarrow\BT_K/\Gamma_0$.
By this and Lemma \ref{lem-finitequot}, we have:

\begin{lemma}
\label{lem-homotopy-type-is-building-modulo-group}
Let $\Gamma$ be a lattice in $\PGL_3(K)$, and $X_{\Gamma}$ the projective variety over $K$ obtained as above.
Then, for any finite field extension $K'/K$, $X^{\Berk}_{\Gamma}\otimes_K K'$ deformation-retracts to the quotient of the geometric realization of $\BT_K$ by $\Gamma$.

In particular, the homotopy type of $X^{\Berk}_{\Gamma}\otimes_K K'$ is
the same as that of the topological space $\BT_K/\Gamma$.
\qed
\end{lemma}

\section{Construction of the fake plane}
\label{sec-construction}

\noindent
We fix $R=\Z_2$ throughout the rest of the
paper and suppress the subscript $K=\Q_2$ from $\Omega^2$,
$\Omegahat^2$ and $\BT$.

We recall the following construction from \cite{AK2011}.
Let $\calO$ be the ring of algebraic integers in $\Q(\sqrt{-7})$,
$\Gamma_L$ be the unitary group of the standard Hermitian lattice
$\calO[\frac{1}{2}]^3$, and $\PGamma_L$ its quotient by scalars.  
To get a lattice in $\PGL_3(\Q_2)$ we
fix an embedding $\calO\to\Z_2$.
This identifies $\PGamma_L$ with a lattice in $\PGL_3(\Q_2)$,
indeed one of the two densest-possible
lattices.
(In \cite{AK2011} we  defined $\Gamma_L$ as the
isometry group of $L[\frac12]:=L\tensor_\calO\calO[\frac{1}{2}]$ for a
more-complicated Hermitian $\calO$-lattice $L$.  But $L[\frac12]=\calO[\frac12]^3$.)



We write $\lambda,\lambdabar$ for $(-1\pm\sqrt{-7})/2$.  These are the two primes
lying over~$2$, and we choose the notation so that $\lambda$
is a
uniformizer of $\Z_2$
and $\lambdabar$ is a 
unit.
Defining $\theta$ as $\lambda-\lambdabar=\sqrt{-7}$, we obtain an
induced inner product on $\calO[\frac12]^3/\theta\calO[\frac12]^3\iso\F_7^3$.   This pairing is  symmetric and nondegenerate, yielding a
natural map from $\Gamma_L$ to the $3$-dimensional orthogonal group over $\F_7$.  This
descends to a homomorphism $\PGamma_L\to\PO_3(7)\iso\PGL_2(7)$.
We write $\Phi_L$ for the kernel.

\begin{lemma}
\label{lem-Phi-L-is-torsion-free}
$\Phi_L\sset\PGL_3(\Q_2)$ is torsion-free.
\end{lemma}

\begin{proof}
We adapt Siegel's proof \cite[\S39]{Siegel} that the kernel of
$\GL_n(\Z)\to\GL_n(\Z/N)$ is torsion-free for any $N\neq2$.  Suppose
given some nontrivial $y\in\Phi_L$ with finite order, which we may
suppose is a rational prime~$p$.  Choose some lift $x\in\Gamma_L$ of it, so
$x^p$ is a scalar $\sigma$.  By $y\in\Phi_L$, the image of $x$ in
${\rm O}_3(7)$ is a scalar, which is to say that $x\cong\pm I$
mod~$\theta$.  We claim: for any $n\geq1$, $x$ is congruent
mod~$\theta^n$ to some scalar $\sigma_n\in\calO[\frac12]$.  It follows easily that
$x$ itself is a scalar, contrary to the hypothesis $y\neq1$.

We prove the claim if $p\neq7$ and
then show how to adapt the argument if $p=7$.  By hypothesis the claim holds
for $n=1$, with $\sigma_1=\pm1$.  For the inductive step suppose
$n\geq1$ and $x\cong\sigma_n I$ mod~$\theta^n$, so $x=\sigma_n
I+\theta^n T$ for some endomorphism $T$ of $\calO[\frac12]^3$.  We must show
that $T$ is congruent mod~$\theta$ to some scalar.  Reducing
$x^p=\sigma$ modulo $\theta^n$ and $\theta^{n+1}$ shows that
$\theta^n$ divides $\sigma-\sigma_n^p$ and that $\sigma_n^p
I+p\sigma_n^{p-1}\theta^n T\cong\sigma I$ mod~$\theta^{n+1}$.
Rearranging shows that $p\sigma_n^{p-1}T$ is the scalar
$(\sigma-\sigma_n^p)/\theta^n$, modulo~$\theta$.  Since $\sigma_n$ and
$p\neq7$ are invertible mod~$\theta$, this gives a formula for $T$
mod~$\theta$, hence $\sigma_{n+1}$ mod~$\theta^{n+1}$, and finishes
the induction.

If $p=7=-\theta^2$ then we write $x=\sigma_n I+\theta^n T$ as before,
but reduce $x^7=\sigma$ modulo $\theta^{n+2}$ and $\theta^{n+3}$
rather than modulo $\theta^n$ and $\theta^{n+1}$.  This shows that
$\theta^{n+2}$ divides $\sigma-\sigma_n^7$ and that $\sigma_n^7
I+7\sigma_n^6\theta^n T\cong\sigma I$ mod~$\theta^{n+3}$.  The rest of
the argument is the same.
\end{proof}

\begin{lemma}
\label{lem-surjection-to-PGL2-7}
$\PGamma_L\to\PGL_2(7)$ is surjective.
\end{lemma}

\begin{proof}
We showed in \cite[Thm.\ 3.2]{AK2011} that $\PGamma_L$ has two orbits on vertices
of $\BT$, with stabilizers $L_3(2)$ and $S_4$.  Fix a vertex $v$ of
the first type.  By lemma~\ref{lem-Phi-L-is-torsion-free}, $\Phi_L$ is torsion-free, so the map
$\PGamma_L\to\PGL_2(7)$ is injective on this $L_3(2)$.  Its image must
be the unique copy of this group in
$\PGL_2(7)$, namely $\PSL_2(7)$.  Next,
the
fourteen subgroups $S_4$ of $L_3(2)$ are the $\PGamma_L$-stabilizers of the
neighbors of $v$.  These are all conjugate in $\PGamma_L$, but not in
$L_3(2)$.  Therefore the image of $\PGamma_L$ in $\PGL_2(7)$ must be
strictly larger than $\PSL_2(7)$, hence equal to $\PGL_2(7)$.
\end{proof}

Since $\Phi_L$ is torsion-free, non-Archimedean uniformization yields
a $\Z_2$-scheme $\mathfrak{X}_{\Phi_L}$.  We will write $\WW_L$ for it
and $W_L$ for its generic fiber.  We fix a Sylow $2$-subgroup of
$\PGL_2(7)$, which is a dihedral group $\dihedral_{16}$ of order~$16$,
and write $\Sigma_L$ for its preimage in $\PGamma_L$.  ($\Sigma$
is meant to suggest Sylow.)  Because $\Phi_L$ is normal in $\Sigma_L$, the
quotient group $\dihedral_{16}$ acts on $\Omegahat^2/\Phi_L$, hence on
$\WW_L$ by the uniqueness of algebraization.  (Indeed all of
$\PGamma_L/\Phi_L=\PGL_2(7)$ acts.)  Because $\WW_L$ is projective and
flat over $\Z_2$, its quotient $\WW_L/\dihedral_{16}$ is also
projective and flat over $\Z_2$.  We write $X_L$ for its generic fiber
$W_L/\dihedral_{16}$.  This is our fake projective plane, proven to be such
in theorem~\ref{thm-ourFPP} below.

The reader familiar with Mumford's construction \cite{Mum1} will recognize
that our constructions parallel his: he considered the projective
unitary group $\Gamma_{\!M}$ of $M[\frac12]$, where $M$ is a different
Hermitian $\calO$-lattice.  He found that its action on
$M[\frac12]/\theta M[\frac12]$ induces a surjection
$\PGamma_M\to\PSL_2(7)$.  The subgroup $\Sigma_M$ of $\PGamma_M$ corresponding to a
Sylow $2$-subgroup $D_8\sset\PSL_2(7)$ uniformizes his fake plane.
His $\Sigma_M$ is torsion-free, while our $\Sigma_L$ contains finite
subgroups~$D_8$.
The following lemma is the key that allows the construction of ``our''
fake plane to work
despite this torsion.  

\begin{lemma}
\label{lem-freeness}
$\dihedral_{16}$ acts freely on the closed points of $W_L$.
In particular, $W_L\to X_L$ is \'etale and $X_L$ is smooth.
\end{lemma}

We remark that $\dihedral_{16}$ has horrible stabilizers in the central
fiber of $\WW_L$, such as components with pointwise stabilizer $(\Z/2)^2$.

\begin{proof}
Recall from section~\ref{sec-uniformization} that the closed points of
$W_L$ are in bijection with the $\Phi_L$-orbits on the 
points of $\Omega^2$. The freeness of $\dihedral_{16}$'s action on this set of
these orbits is equivalent to the freeness of $\Sigma_L$'s action on
the closed points of $\Omega^2$.  Since $\Phi_L$ is a torsion-free
uniform lattice, it is normal hyperbolic in the sense of \cite[\S
  1]{Mus1}, so it acts freely on $\Omega^2$.  An infinite-order
element of $\Sigma_L$ cannot have fixed points in $\Omega^2$, because
some power of it is a nontrivial element of $\Phi_L$.  So 
only the torsion elements of $\Sigma_L$ could have fixed points.

Because $\Phi_L$ is torsion-free, the
map $\Sigma_L\to\Sigma_L/\Phi_L=\dihedral_{16}$ preserves the orders of
torsion elements.  Therefore every
torsion element of $\Sigma_L$ has $2$-power order.  To show that none
of them have fixed points in $\Omega^2$, it suffices to show that
no involution in $\Sigma_L$ has a fixed point.
In fact, no involution in $\PGL_3(\Q_2)$ has a fixed point in
$\Omega^2$: every involution lifts to an involution in $\GL_3(\Q_2)$, whose
eigenspaces are defined over $\Q_2$, hence were removed from
$\P^{2,\an}_K$ in the definition of $\Omega^2$.
\end{proof}

\begin{theorem}
\label{thm-ourFPP}
$X_L$ is a fake projective plane.
\end{theorem}

\begin{proof}
First we count sixteen $\Phi_L$-orbits on vertices of $\BT$:
the $\PGamma_L$-orbit of vertices with stabilizer $L_3(2)$
splits into 
$[\PGL_2(7):L_3(2)]=2$ orbits under $\Phi_L$, and the
$\PGamma_L$-orbit of vertices with stabilizer $S_4$ 
splits into $[\PGL_2(7):S_4]=14$ orbits under $\Phi_L$.  

So proposition~\ref{prop-invariants} shows that $\chi(W_L)=16$,
$q(W_L)=0$, $c^2_1(W_L)=3c_2(W_L)=144$, and that $W_L$ has ample
canonical class.  We now use three times the fact that $W_L\to X_L$ is
\'etale.  First, since the degree is~$16$, we have $\chi(X_L)=1$ and
$c^2_1(X_L)=3c_2(X_L)=9$.  Second, $X_L$ has the same Kodaira
dimension as $W_L$ (e.g.\ \cite[Chap.\ I, (7.4)]{BPV}), hence has
general type.  
Third, since $W_L$ has irregularity~$0$, the following lemma shows
that $X_L$ also has irregularity~$0$.
From the definition of $\chi$ it follows that $P_g(X_L)=0$,
completing the proof.
\end{proof}

\begin{lem}\label{lem-etaleirregularityzero}
Let $X$ and $Y$ be algebraic varieties over a field $K$, and $f\colon Y\rightarrow X$ a finite flat morphism of degree not divisible by the characteristic of $K$.
Let $q>0$ be a positive integer.
Then, if $\H^q(Y,\calO_Y)=0$, we have $\H^q(X,\calO_X)=0$.
\end{lem}

\begin{proof}
By flatness,  $f_{\ast}\calO_Y$ is locally free on $X$. Then the trace map
$\mathrm{tr}_{Y/X}\colon f_{\ast}\calO_Y\rightarrow\calO_X$, divided
by the degree of $f$, gives a splitting of the inclusion
$\calO_X\hookrightarrow f_{\ast}\calO_Y$.  Since $\calO_X$ is a direct summand
of $f_{\ast}\calO_Y$, the lemma follows immediately.
\end{proof}

\begin{remark}
The fake projective plane $X_L$ is commensurable with the Mumford's
$X_M$, by \cite[Theorem 3.3]{AK2011}.
\end{remark}

\section{Distinctness from other fake planes}
\label{sec-distinctness}

\noindent
Our final result is the following:

\begin{theorem}
The fake plane $X_L$ is not isomorphic over $\Qbar_2$ to any fake
plane uniformized by a torsion-free subgroup of $\PGL_3(\Q_2)$.
\end{theorem}

\begin{proof}
Suppose $X$ is a fake projective plane uniformized by a torsion-free
subgroup $\PGamma$ of $\PGL_3(\Q_2)$.  Although we don't need it, we
remark that there are three possibilities: Mumford's example and two
due to Ishida--Kato\cite{IK98}.  By lemma~\ref{lem-homotopy-type-is-building-modulo-group}, the Berkovich space
$X^\Berk$ has the homotopy type of $\BT/\PGamma$.  Since $\BT$ is
contractible and $\Gamma$ acts freely (by the absence of torsion), the
fundamental group of $X^\Berk$ is isomorphic to $\PGamma$.
Furthermore, lemma~\ref{lem-homotopy-type-is-building-modulo-group}
assures us that the base extension $X^\Berk\tensor_{\Q_2}K'$ also
has fundamental group $\PGamma$, for
any finite extension $K'$ of $\Q_2$. 

Repeating the argument shows that $X_L^\Berk\tensor_{\Q_2}K'$ is
homotopy equivalent to $\BT/\Sigma_L$, for any finite extension $K'$
of $\Q_2$.  And  lemma~\ref{lem-pi-1-is-Z-mod-42} below shows that $\BT/\Sigma_L$ has the
homotopy type of the standard presentation complex of $\Z/42$.   That
is, a circle with a disk attached by wrapping the boundary of the disk
$42$ times around the circle.  It follows that
$X_L^\Berk\tensor_{\Q_2}K'$ has fundamental group $\Z/42$.

If $X$ and $X_L$ were isomorphic over $\Qbar_2$ then they would
be isomorphic over some finite extension $K'$ of $\Q_2$.  Then the
isomorphism $X\tensor_{\Q_2}K'\iso X_L\tensor_{\Q_2}K'$ would imply
$X^\Berk\tensor_{\Q_2}K'\iso X^\Berk_L\tensor_{\Q_2}K'$.  But this is
impossible since the left side has infinite
fundamental group and the right side has fundamental group $\Z/42$.
\end{proof}

It remains to prove lemma~\ref{lem-pi-1-is-Z-mod-42}, describing the homotopy type of
$\BT/\Sigma_L$.  The rest of this section is devoted to this.  The key
is understanding the central fiber of $\WW_L/\dihedral_{16}$, which in turn requires
understanding the central fiber of $\WW_L$.  Recall that the central
fiber of $\Omegahat^2$ is a normal crossing divisor with properties
described in section~\ref{sec-uniformization}.

The central fiber of $\WW_L$ is normal crossing because it is the
quotient of the central fiber of $\Omegahat^2$ by the group $\Phi_L$
acting freely.  To describe it we need to enumerate its components,
double curves and triple points.  Our description in the next lemma
refers to the ``elements'' of the finite projective geometry, meaning
the seven points and seven lines of the projective plane over $\F_2$.
We regard these as the vertices of a graph, with two elements incident
if one corresponds to a point and the other to a line containing it.
The symbols $e,f$ will always refer to such ``elements'', and the
symbols $p,p',p''$ (resp.\ $l,l',l''$) will always refer to points
(resp.\ lines) of this geometry.  The automorphism group of the graph
is $\PGL_2(7)\iso\PSL_2(7)\semidirect(\Z/2)\iso\GL_3(2)\semidirect(\Z/2)$.  Classically, the
elements of $\PGL_2(7)$ not in $\PSL_2(7)$ are called
``correlations''; they exchange points and lines.

\begin{lemma}
\label{lem-central-fiber}
$\WW_{L,0}$ has $16$ components, $112$ double curves
and $112$ triple points.  In more detail,
\leavevmode\hbox{}
\begin{enumerate}
\item
\label{item-names-of-components}
We may label $\WW_{L,0}$'s components
$\Pi$, $\Pi^*$ and
$C_e$, such that 
$\PGL_2(7)$ permutes the $C_e$'s the
same way it permutes the elements $e$ of the finite geometry, and 
correlations exchange
$\Pi$ and~$\Pi^*$.
\item
\label{item-Pi-and-Pi-star-disjoint}
$\Pi$ and $\Pi^*$ are disjoint.
\item
\label{item-how-C-e-meets-Pi-and-Pi-star}
$D_e:=C_e\cap\Pi$ and $D_e^*:=C_e\cap\Pi^*$ are irreducible curves.
\item
\label{item-C-e-disjoint-from-C-f-non-incident-case}
If $e\neq f$ are not incident then $C_e\cap C_f=\emptyset$.
\item
\label{item-how-C-e-meets-C-f-incident-case}
If $e,f$ are incident then $C_e\cap C_f$ has two components.  One,
which we call $D_{e f}$, has 
self-intersection $-1$ in
$C_e$ and $-2$ in $C_f$.  The other, called $D_{f e}$, has these
numbers reversed.
\item
\label{item-double-curves-of-C-e}
The singular locus of each $C_e$ is a curve of three components.  For
each $f$ incident to~$e$, exactly one of these components meets $C_f$;
we call it $E_{e f}$.
\item
\label{item-triple-points-lying-in-Pi-and-Pi-star}
If $e,f$ are incident then each of $P_{e f}:=\Pi\cap C_e\cap C_f$ and
$P_{e f}^*:=\Pi^*\cap C_e\cap C_f$ is a single point.  
\item
\label{item-Q-e-f}
If $e,f$ are incident then $Q_{e f}:=E_{e f}\cap C_f$ is a single point.
\item
\label{item-self-triple-intersections}
Each $C_e$ has two triple-self-intersection points.  At such a triple
point the incident double curves are $E_{e f_1}$, $E_{e f_2}$ and
$E_{e f_3}$ where $f_1,f_2,f_3$ are the elements of the geometry
incident to $e$.  We may label these triple points $R_{e o}$, where $o$ is a
cyclic ordering on $\{f_1,f_2,f_3\}$, such that $\PGL_2(7)$ permutes
them the same way it permutes the ordered pairs $(e,o)$.
\end{enumerate}
\renewcommand{\theenumi}{\alph{enumi}}
The components fall into two $\PGL_2(7)$-orbits:
\begin{enumerate}
\item
$\{\Pi,\Pi^*\}$
\item
the fourteen $C_e$'s.
\end{enumerate}
The double curves fall into four $\PGL_2(7)$-orbits:
\begin{enumerate}
\item
the 
seven $D_p$'s and seven $D_l^*$'s
\item the seven
$D_l$'s and seven $D_p^*$'s
\item
the forty-two $D_{e f}$'s
\item
the forty-two $E_{e f}$'s.
\end{enumerate}
The triple points fall into three $\PGL_2(7)$-orbits:
\begin{enumerate}
\item the twenty-one
$P_{e f}$'s and twenty-one $P_{e f}^*$'s
\item 
the forty-two $Q_{e f}$'s
\item the twenty-eight $R_{e o}$'s.
\end{enumerate}
\end{lemma}

Note that $P_{e f}=P_{f e}$ and $P_{e f}^*=P_{f e}^*$, unlike all other cases
involving double subscripts.  

\begin{proof}
By \cite[Thm.\ 3.2]{AK2011}, $\PGamma_L$ acts on the vertices of $\BT$ with two
orbits, having stabilizers $L_3(2)$ and $S_4$.  We will pass between
vertices of $\BT$ and components of $\Omegahat^2$ without comment
whenever it is convenient.  Write $\Pitilde$ for a component of
$\Omegahat^2$ with stabilizer $L_3(2)$.  Recall that $\Phi_L$ is the
kernel of a surjection $\PGamma_L\to\PGL_2(7)$.  Since $\Phi_L$ is
torsion-free, $L_3(2)$ must inject into $\PGL_2(7)$, so its
$\PGamma_L$-orbit splits into two $\Phi_L$-orbits.  We write $\Pi$ and
$\Pi^*$ for them, and also for the corresponding components of $\WW_{L,0}$.  The same argument shows that the $\PGamma_L$-orbit
with stabilizer $S_4$ splits into $[\PGL_2(7):S_4]=14$ orbits under
$\Phi_L$.  Because there is only one conjugacy class of $S_4$'s in
$\PGL_2(7)$, the action of $\PGL_2(7)$ on these components of
$\WW_{L,0}$ must
correspond to the action on the elements of the finite geometry.  We
have proven~\eqref{item-names-of-components}.  We will call the
components other than $\Pi,\Pi^*$ the {\it side components}; this
reflects our mental image of $\WW_{L,0}$: $\Pi$
above, $\Pi^*$ below, and the other components around the sides.

By the explicit description of $\PGamma_L$ in the proof of theorem~3.2
of \cite{AK2011}, each of $\Pitilde$'s neighbors in
$\BT$ has $\PGamma_L$-stabilizer $S_4$, hence is inequivalent to
$\Pitilde$.  Therefore the union of the $\PGamma_L$-translates of
$\Pitilde$ is the disjoint union of its components.  Since $\Phi_L$
permutes these components freely, it follows
that $\Pi$ and $\Pi^*$ are disjoint, proving
\eqref{item-Pi-and-Pi-star-disjoint}.  It also follows  that $\Pitilde$ maps
isomorphically to $\Pi$.

Therefore $\Pi$ is a copy of $\P^2_{\F_2}$ blown up at its seven
$\F_2$-points.  The curves along which it meets other components are
the seven exceptional divisors and the strict transforms of the seven
$\F_2$-rational lines.  Suppose $e$ is the element of the finite
geometry corresponding to one of these curves.  The
$L_3(2)$-stabilizer of $e$ preserves exactly one element of the
geometry, namely $e$, hence exactly one side component, namely $C_e$.
So $C_e$ must be the side component that meets $\Pi$ along the chosen
curve.  In this way the $14$ side components account for all the
double curves lying in $\Pi$, proving that each $C_e\cap\Pi$ is
irreducible.  By symmetry the same holds for $C_e\cap\Pi^*$.  This
proves 
\eqref{item-how-C-e-meets-Pi-and-Pi-star}, and then 
\eqref{item-triple-points-lying-in-Pi-and-Pi-star} is immediate.

A simple counting argument shows that $\WW_{L,0}$ has $112$ double
curves and $112$ triple points.  We have already named the $28$ double
curves ($D_e$ and $D_e^*$) 
that lie in $\Pi$ or $\Pi^*$, leaving~$84$.
We observe that
if two side components meet then their
intersection consists of an even number of components.  This is
because for any elements $e\neq f$ of the geometry, there is some
$g\in\PGL_2(7)$ exchanging them.  So if a component of $C_e\cap C_f$ has
self-intersection $-1$ in $C_e$ and $-2$ in $C_f$, then its $g$-image has
these self-intersection numbers reversed, and therefore cannot be the
same curve.  

If $e,f$ are incident then $C_e\cap C_f$ contains $P_{e f}$ and is
therefore non\-empty.  By the previous paragraph it has evenly
many components.  Because there are $21$ unordered incident pairs
$e,f$, this accounts for either $42$ or $84$ of the $84$ remaining double
curves, according to whether $C_e\cap C_f$ has $2$ or $4$ components.
We will see later that they account for~$42$ of them.

We claim next that if $e$ and $f$ are a
point and a nonincident line, then $C_e\cap C_f=\emptyset$.  This is
because such pairs $\{e,f\}$ form a $\PGL_2(7)$-orbit of size $28$.
If $C_e\cap C_f\neq\emptyset$ then the argument from the previous paragraph shows that 
such intersections account
for at least $56$ double curves, while at most $42$
remain unaccounted for.  The same argument shows $C_e\cap C_f=\emptyset$ if
$e,f$ are distinct lines or distinct points.  This proves
\eqref{item-C-e-disjoint-from-C-f-non-incident-case}. 

Consider one of the $112-42=70$ triple points outside $\Pi\cup\Pi^*$
and the three (local) components of $\WW_{L,0}$ there.  Two of these
have the same type (i.e., both correspond to points or both to lines).
Since they meet, the previous paragraph shows that these components
must coincide.  It follows that each side-component has at least one
curve of self-intersection.  We saw above that if $e,f$ are incident
then $C_e\cap C_f$ has either two or four components, and in the
latter case these intersections account for all double curves not in
$\Pi\cup\Pi^*$.  Therefore this case is impossible, proving 
\eqref{item-how-C-e-meets-C-f-incident-case}.  Now
\eqref{item-Pi-and-Pi-star-disjoint}--\eqref{item-how-C-e-meets-C-f-incident-case}
show that every one of the $112-28-42=42$ remaining double curves is a
self-intersection curve of a side component.  So each side component
contains $42/14=3$ such curves, proving the first part of
\eqref{item-double-curves-of-C-e}.

Next we claim that there exist incident $e,f$ such that there is a
triple point where two of the (local) components are $C_e$ and the
third is $C_f$.  To see this choose any incident $e,f$ and recall from
\eqref{item-triple-points-lying-in-Pi-and-Pi-star} that $C_e\cap C_f=D_{e f}\cup D_{f e}$ meets
$\Pi\cup\Pi^*$ exactly twice.  So it must contain some other triple
point.  By our understanding of double curves the third component
there must be $C_e$ or $C_f$.  After exchanging $e$ and $f$ if necessary,
this proves the claim.  It follows by symmetry that for any incident
$e,f$ there is such a triple point, and in fact exactly one since
there are $42$ ordered incident pairs $e,f$ and only $70$
triple points outside $\Pi\cup\Pi^*$.   It follows from this
uniqueness that exactly one of the three self-intersection curves of
$C_e$ meets $C_f$, and it does so at a single point.  This proves the
second half of \eqref{item-double-curves-of-C-e} and 
all of \eqref{item-Q-e-f}.

The remaining $112-42-42=28$ triple points must all be
triple-self-intersections of the $C_e$'s, so each $C_e$ contains two
of them.  Now fix $e$ and write $\tau$ and $\tau'$ for these
self-intersection points.  Obviously the only double-curves that can
pass through $\tau$ or $\tau'$ are $E_{e f_1}$, $E_{e f_2}$ and $E_{e f_3}$.
The $S_4\sset\PGL_2(7)$ fixing $e$ contains an element of order~$3$
cyclically permuting $f_1,f_2,f_3$, necessarily 
fixing each of $\tau,\tau'$.  It follows that each of $\tau,\tau'$ lies
in all three of the $E_{e f_i}$.  Therefore each determines a cyclic
ordering on $\{E_{e f_1},E_{e f_2},E_{e f_3}\}$, hence on $\{f_1,f_2,f_3\}$.  
Since $S_4$ acts on the $E_{e f_i}$ as $S_3$, both
cyclic ordering occur, and it follows that $\tau,\tau'$ induce
the two possible cyclic orderings.  This proves
\eqref{item-self-triple-intersections}.
\end{proof}

Translating the lemma into the dual-complex language gives a complete
description of the dual complex of $\WW_{L,0}$:

(1) Its
vertices are $\Pi$, $\Pi^*$ and the $C_e$.  

(2) For each $p$, there is an edge $D_p$ from $\Pi$ to $C_p$ and an
edge $D_p^*$ from $C_p$ to $\Pi^*$.

(3) For each $l$ there is an edge $D_l^*$ from $\Pi^*$ to $C_l$  and an edge
$D_l$ from $C_l$ to  $\Pi$.

(4) For each ordered pair $(e,f)$ with $e,f$ incident, there is an
edge $D_{e f}$ from $C_e$ to $C_f$ and an edge $E_{e f}$ from $C_e$ to itself.

(5) For each point $p$ and line $l$ that are incident, there is a
$2$-cell $P_{pl}=P_{l p}$ with its boundary attached along the loop
$D_p.D_{pl}.D_l$, and a $2$-cell $P_{pl}^*=P_{l p}^*$ with its
boundary attached along the loop $D_l^*.D_{l p}.D_p^*$.

(6) For each ordered pair $(e,f)$ with $e,f$ incident, there is a
$2$-cell $Q_{e f}$ with its boundary attached along the loop
$D_{e f}.D_{f e}.E_{e f}$. 

(7) For each $e$, there are $2$-cells $R_{e o}$ and $R_{e o'}$ where
$o,o'$ are the two cyclic orderings on $\{f_1,f_2,f_3\}$.  Their
boundaries are attached along the loops $E_{e f_1}.E_{e f_2}.E_{e f_3}$
and $E_{e f_3}.E_{e f_2}.E_{e f_1}$.

\medskip
Really we are interested in the complex $\BT/\Sigma_L$, which is the
same as the quotient of the complex just described by the dihedral
group $\dihedral_{16}$.  It is easy to see that if an element of $\dihedral_{16}$
fixes setwise one of the cells just listed, then it fixes it
pointwise.  Therefore $\BT/\Sigma_L$ is a CW complex with one cell for
each $\dihedral_{16}$-orbit of cells of $\BT/\Phi_L$.  To tabulate these
orbits we note that $\dihedral_{16}$ contains correlations, so $\Pi$ and
$\Pi^*$ are equivalent, and every $C_l$ is equivalent to some $C_p$.
Next, the subgroup $D_8$ sending points to points and lines to lines
is the flag stabilizer in $L_3(2)$.  So it acts on the points
(resp.\ lines) with orbits of sizes $1$, $2$ and $4$.  We write
$p,p',p''$ (resp.\ $l,l',l''$) for representatives of these orbits.
Since $\dihedral_{16}$ normalizes $D_8$, the correlations in it exchange the
orbit of points of size $1$, resp.\ $2$, resp.\ $4$ with the orbit of
lines of the same size.  That is, $p$, resp.\ $p'$, resp.\ $p''$ is
$\dihedral_{16}$-equivalent to $l$, resp.\ $l'$, resp.\ $l''$.

\def\Pibar{\overline{\Pi}}
\def\Cbar{C\llap{$\overline{\phantom{\rm C}}$}}
\def\Dbar{D\llap{$\overline{\phantom{\rm D}}$}}
\def\Ebar{E\llap{$\overline{\phantom{\rm E}}$}}
\def\Pbar{P\llap{$\overline{\phantom{\rm P}}$}}
\def\Qbar{Q\llap{$\overline{\phantom{\rm Q}}$}}
\def\Rbar{R\llap{$\overline{\phantom{\rm R}}$}}

\begin{lemma}
\label{lem-cells-of-central-fiber-of-X-L}
$\BT/\Sigma_L$ is the CW complex with four vertices $\Pibar$, $\Cbar_p$,
$\Cbar_{p'}$
and $\Cbar_{p''}$, and higher-dimensional cells as follows.  Its $18$ edges
are
\smallskip
\begin{center}
\def\itself{\rm itself}
\begin{tabular}{lccccccccc}%
    &$\Dbar_p$&$\Dbar_{p'}$&$\Dbar_{p''}$&$\Dbar_{p p}$&$\Dbar_{p p'}$&$\Dbar_{p' p}$&$\Dbar_{p' p''}$&$\Dbar_{p''p'}$&$\Dbar_{p''p''}$\\
from&$\Pibar$&$  \Pibar$&$   \Pibar$&$\Cbar_{p} $&$\Cbar_{p} $&$\Cbar_{p'} $&$\Cbar_{p'}  $&$\Cbar_{p''}  $&$\Cbar_{p''}$\\
to&$  \Cbar_p$&$\Cbar_{p'}$&$\Cbar_{p''}$&\itself&$\Cbar_{p'} $&$\Cbar_{p}  $&$\Cbar_{p''} $&$\Cbar_{p'}   $&\itself\\
\noalign{\smallskip}%
\cline{2-10}%
\noalign{\smallskip}%
    &$\Dbar_p^*$&$\Dbar_{p'}^*$&$\Dbar_{p''}^*$&$\Ebar_{p p}$&$\Ebar_{p p'}$&$\Ebar_{p' p}$&$\Ebar_{p' p''}$&$\Ebar_{p''p'}$&$\Ebar_{p''p''}$\\
from&$\Cbar_p  $&$\Cbar_{p'}  $&$\Cbar_{p''} $&$\Cbar_{p} $&$\Cbar_{p}  $&$\Cbar_{p'} $&$\Cbar_{p'}  $&$\Cbar_{p''}  $&$\Cbar_{p''}$\\
to&$  \Pibar  $&$\Pibar    $&$\Pibar    $&\itself&\itself &\itself &\itself &\itself    &\itself
\end{tabular}
\end{center}
\par\smallskip\noindent
Its $15$ two-cells and their boundaries are
\medskip
\begin{center}
\setlength{\tabcolsep}{0pt}
\begin{tabular}{llcll}
$\Pbar_{p p}$&${}:\Dbar_p.\Dbar_{p p}.\Dbar_p^*$&\kern20pt
&$\Pbar_{p p'}$&${}:\Dbar_p.\Dbar_{p p'}.\Dbar_{p'}^*$\\
$\Pbar_{p' p}$&${}:\Dbar_{p'}.\Dbar_{p' p}.\Dbar_p^*$&\kern20pt
&$\Pbar_{p' p''}$&${}:\Dbar_{p'}.\Dbar_{p' p''}.\Dbar_{p''}^*$\\
$\Pbar_{p''p'}$&${}:\Dbar_{p''}.\Dbar_{p''p'}.\Dbar_{p'}^*$&\kern20pt
&$\Pbar_{p''p''}$&${}:\Dbar_{p''}.\Dbar_{p''p''}.\Dbar_{p''}^*$\\
$\Qbar_{p p}   $&${}:\Dbar_{p p}.\Dbar_{p p}.\Ebar_{p p}$&\kern20pt
&$\Qbar_{p p'}  $&${}:\Dbar_{p p'}.\Dbar_{p' p}.\Ebar_{p p'}$\\
$\Qbar_{p' p}  $&${}:\Dbar_{p' p}.\Dbar_{p p'}.\Ebar_{p' p}$&\kern20pt
&$\Qbar_{p' p''} $&${}:\Dbar_{p' p''}.\Dbar_{p''p'}.\Ebar_{p' p''}$\\
$\Qbar_{p''p'} $&${}:\Dbar_{p''p'}.\Dbar_{p' p''}.\Ebar_{p''p'}$&\kern20pt
&$\Qbar_{p''p''}$&${}:\Dbar_{p''p''}.\Dbar_{p''p''}.\Ebar_{p''p''}$\\
$\Rbar_p$&${}:\Ebar_{p p}.\Ebar_{p p'}^2$&\kern20pt
&$\Rbar_{p'}$&${}:\Ebar_{p' p}.\Ebar_{p' p''}^2$\\
$\Rbar_{p''}$&${}:\Ebar_{p''p'}.\Ebar_{p''p''}^2$
\end{tabular}
\end{center}
\end{lemma}

\begin{proof}
The remarks above show that the $\dihedral_{16}$-orbits on vertices of
$\BT/\Phi_L$ have representatives $\Pi$, $C_p$, $C_{p'}$, $C_{p''}$.
We add a bar to indicate their images, the vertices of $\BT/\Sigma_L$.  

By the presence of correlations, the edges $D_e$ and $D_e^*$ with $e$
a line are $\dihedral_{16}$-equivalent to edges $D_f^*$ and $D_f$ with $f$ a
point.  Therefore orbit representatives for the $\dihedral_{16}$-action on
the $28$ edges listed under (2) and (3) are $D_p$, $D_{p'}$,
$D_{p''}$, $D_p^*$, $D_{p'}^*$,
$D_{p''}^*$.  We add a bar to indicate their images.   

Again using the presence of correlations, the $\dihedral_{16}$-orbits of
ordered pairs $(e,f)$ with $e$ and $f$ incident are in bijection with
the $D_8$-orbits of such pairs in which $e$ is a point.  
These
$D_8$-orbits are represented by 
\begin{equation}
\label{eq-list-of-pairs}
(p,l),\ (p,l'),\ (p',l),\ (p',l''),\ (p'',l')\ \hbox{and}\ (p'',l''),
\end{equation}
which therefore index the six
$\dihedral_{16}$-orbits on the $42$ edges $D_{e f}$ (resp.\ $E_{e f}$).  
The edges $D_{pl}$, $D_{pl'}$, $D_{p' l}$,
$D_{p' l''}$, $D_{p''l'}$ and $D_{p''l''}$ go from $C_p$ to $C_l$,
$C_p$ to $C_{l'}$, $C_{p'}$ to $C_l$, etc.  Therefore their images go
from $\Cbar_p$ to itself, $\Cbar_p$ to $\Cbar_{p'}$, $\Cbar_{p'}$ to $\Cbar_p$, etc.  We
call the images $\Dbar_{p p}$, $\Dbar_{p p'}$, $\Dbar_{p' p}$,
etc.  
The
edges $E_{pl}$, $E_{pl'}$, $E_{p' l}$, $E_{p' l''}$, $E_{p''l'}$ and
$E_{p''l''}$ are loops based at $C_p$, $C_p$, $C_{p'}$, $C_{p'}$,
$C_{p''}$ and~$C_{p''}$.  We call their images $\Ebar_{p p}$,
$\Ebar_{p p'}$, $\Ebar_{p' p}$, etc.

The two-cells $P_{pl}$ meet $\Pi$ but not $\Pi^*$, while the
$P_{pl}^*$ meet $\Pi^*$ but not $\Pi$.  Therefore the
$\dihedral_{16}$-orbits on these cells are in bijection with the
$D_8$-orbits on the $P_{pl}$.  As in the previous paragraph,
orbit representatives are $P_{pl}$, $P_{pl'}$, $P_{p' l}$, $P_{p' l''}$,
$P_{p''l'}$ and $P_{p''l''}$.  
We call their images $\Pbar_{p p}$, $\Pbar_{p p'}$, etc., and their
attaching maps are easy to work out.  For example, the boundary of
$P_{pl'}$ is given above as $D_p.D_{pl}.D_l$.  The images of the first
two terms are $\Dbar_p$ and $\Dbar_{p p}$, and $D_l$ is
$\dihedral_{16}$-equivalent to $D_p^*$, so the image of the third term is
$\Dbar_p^*$.   Therefore the boundary of the disk $\Pbar_{pl'}$ is
attached along $\Dbar_p.\Dbar_{p p}.\Dbar_p^*$.

In the same way, $\dihedral_{16}$-orbit representatives on the $2$-cells
$Q_{e f}$ are $Q_{pl}$, $Q_{pl'}$, $Q_{p' l}$, $Q_{p' l''}$, $Q_{p''l'}$
and $Q_{p''l''}$.  We indicate their images in a similar way to the
other images: we add a bar
and convert subscript $l$'s to $p$'s.  As an example we work out
the boundary of $\Qbar_{p' p''}$, using the boundary of $Q_{p' l''}$
given above as $D_{p' l''}.D_{l''p'}.E_{p' l''}$.  The images of the
first and third terms are $\Dbar_{p' p''}$ and $\Ebar_{p' l''}$.  For
the image of the second term, we apply a correlation sending $l''$ to
$p''$.  So the ordered pair $(l'',p')$ is $\dihedral_{16}$-equivalent to
some ordered pair $(p'',m)$ where $m$ is a line incident to $p''$ and
$D_8$-equivalent to $l'$.  This is $D_8$-equivalent to some pair from
\eqref{eq-list-of-pairs}, and $(p'',l')$ is the only possibility.  
Therefore
$D_{l''p'}$ is $\dihedral_{16}$-equivalent to $D_{p''l'}$, so
the boundary of $\Qbar_{p' p''}$ is
$\Dbar_{p' p''}.\Dbar_{p''p'}.\Ebar_{p' p''}$.   The other cases are
essentially the same.

For the $\dihedral_{16}$-orbits on the $2$-cells $R_{e o}$ we note that each
of $p,p',p''$ is fixed by an element of $D_8$ that exchanges two of
the three incident lines.  It follows the $\dihedral_{16}$-orbit
representatives on these $2$-cells are $R_{p o}$, $R_{p' o'}$ and
$R_{p''o''}$, where $o$ (resp. $o'$, $o''$) is a fixed cyclic ordering
on the three lines incident to $p$ (resp. $p'$, $p''$).  We write
$\Rbar_p$, $\Rbar_{p'}$ and $\Rbar_{p''}$ for their images.  Their
boundary maps can be worked out using the following.  The three lines
through $p$ are $l$, $l'$, and another line which is $D_8$-equivalent
to $l'$.  The three pairs $(p',m)$, with $m$ a line through $p'$, are
$D_8$-equivalent to $(p',l)$, $(p',l'')$ and $(p',l'')$.  The three
pairs $(p'',m)$ with $m$ a line through $p''$, are $D_8$-equivalent to
$(p'',l')$, $(p'',l'')$ and $(p'',l'')$.  It follows that the
boundaries of $\Rbar_p$, $\Rbar_{p'}$ and $\Rbar_{p''}$ are attached
along the stated loops.
\end{proof}

\begin{lemma}
\label{lem-pi-1-is-Z-mod-42}
$\BT/\Sigma_L$ is homotopy-equivalent to the standard presentation
complex of $\Z/42$.  In particular, its  fundamental group is $\Z/42$.
\end{lemma}

\begin{proof}
To simplify matters we build up the $2$-complex in several stages,
suppressing the bars from the names of cells to lighten the notation.
First we define $K_1$ as the $1$-complex with the $4$ vertices and the
edges $D_p,D_{p'},D_{p''},D_p^*,D_{p'}^*,D_{p''}^*$.  We collapse the
last three edges to points, leaving a rose with three petals, which we
will call $K_2$.  If the boundary of a $2$-cell to be attached later
involves one of the collapsed edges then we also collapse that portion
of the boundary (i.e., we may ignore it).

We define $K_3$ by attaching to $K_2$ the edges
\begin{equation}
\label{eq-the-D-s}
D_{p p},D_{p p'},D_{p' p},D_{p' p''},D_{p''p'},D_{p''p''}
\end{equation}
(which are
loops in $K_2$) and the $2$-cells $P_{\ast\ast}$ 
having the same
subscripts.  We may deformation-retract $K_3$ back to $K_2$ because
each of the newly-adjoined edges is involved in only one of the
$2$-cells.  In particular, the loops \eqref{eq-the-D-s} are homotopic
rel basepoint to the inverses of $D_p$, $D_p$, $D_{p'}$, $D_{p'}$,
$D_{p''}$ and $D_{p''}$.  

We define $K_4$ by attaching to $K_2$ the edges
\begin{equation}
\label{eq-the-E-s}
E_{p p},E_{p p'},E_{p' p},E_{p' p''},E_{p''p'},E_{p''p''}
\end{equation}
and the $2$-cells $Q_{\ast\ast}$
having the same subscripts.  Just as
for $K_3$, we may deformation-retract $K_4$ back to $K_2$.  The loops
\eqref{eq-the-E-s} are homotopic rel basepoint to $D_p^2$,
$D_{p'}D_p$, $D_{p}D_{p'}$, $D_{p''}D_{p'}$, $D_{p'}D_{p''}$ and
$D_{p''}^2$.

Finally we define $K_5$ by attaching the cells $R_{p}$, $R_{p'}$,
$R_{p''}$ to $K_2$.  $\BT/\Sigma_L$ is homotopy-equivalent to this, hence to
the rose with three petals $D_p$, $D_{p'}$, $D_{p''}$ with three disks
attached, along the curves $D_p^2\bigl(D_{p'}D_p\bigr){}^2$, $D_p
D_{p'}\bigl(D_{p''}D_{p'}\bigr){}^2$ and
$D_{p'}D_{p''}\bigl(D_{p''}D_{p''}\bigr){}^2$.  Regarding these as
relators defining $\pi_1(\BT/\Sigma_L)$, the third one
allows us to eliminate $D_{p'}$ and replace it by $D_{p''}^{-5}$.
Then the second one allows us to eliminate $D_p$ and replace it
by $D_{p''}^{13}$.  The remaining relation then reads $D_{p''}^{42}=1$.
\end{proof}

\frenchspacing
\begin{small}

\end{small}

\end{document}